\documentclass[12pt,reqno]{amsart}
\usepackage{amsmath,amssymb}
\usepackage{cite}

 \makeatletter
\evensidemargin \oddsidemargin
\makeatother

 \newtheorem{theorem}{Theorem}[section]
\newtheorem{lemma}[theorem]{Lemma}
\newtheorem{proposition}[theorem]{Proposition}

\theoremstyle{definition}

\theoremstyle{remark}
\newtheorem{remark}[theorem]{Remark}
\numberwithin{equation}{section}

 \newtheorem*{theorem*}{Theorem}

\numberwithin{equation}{section}
\numberwithin{equation}{section}

\begin{document}
\title[ Some Hadamard product inequalities for accretive matrices ]{ Some Hadamard product inequalities for accretive matrices }

 \author[A. Sheikhhosseini, S. Malekinejad and M. Khosravi]{Alemeh
Sheikhhosseini$^{1}$, Somayeh Malekinejad$^{2*}$ and Maryam Khosravi$^3$}
\address{ $^{1}$ Department of Pure Mathematics\\
Faculty of Mathematics and Computer, Shahid Bahonar University of Kerman,\\
	Kerman, Iran.}
\email{sheikhhosseini@uk.ac.ir;hosseini8560@gmail.com}
\address{ $^{2}$ Department of Mathematics\\
Payame Noor University, P.O. Box 19395-3697\\
Tehran, Iran.}
\email{maleki60313@pnu.ac.ir; maleki60313@gmail.com}
\address{ $^{3}$ Department of Pure Mathematics\\
Faculty of Mathematics and Computer, Shahid Bahonar University of Kerman,\\
	Kerman, Iran.}
\email{khosravi$_-$m @uk.ac.ir}

 \thanks{*Corresponding author}

\setcounter{section}{0}
\numberwithin{equation}{section}
\begin{abstract}\noindent
In this paper, we obtain some new matrix inequalities involving Hadamard product. Also some Hadamard product inequalities for accretive matrices involving the matrix means, positive unital linear maps and matrix concave functions are investigated.
Among other results, it is shown that if $A, B, C, D$ are $n\times n$ positive definite matrices, then
\begin{equation*}
\left(\alpha A+\beta B\right)^r\circ\left(\alpha C+\beta D\right)^{1-r}\leq \alpha\left(A^r\circ C^{1-r}\right)+\beta\left(B^r\circ D^{1-r}\right),
\end{equation*}
where $r \in (-1, 0) \cup (1, 2)$ and $" \circ "$ stands for the Hadamard product.

 \end{abstract}

\maketitle

 \section{Introduction and preliminaries}
Let $\mathbb{M}_{m\times n}$ be the space of all $m\times n$ matrices with entries in the complex field $\mathbb{C}.$
If $m=n,$ then $\mathbb{M}_{n\times n}$ is denoted by $\mathbb{M}_n.$
A Hermitian matrix $A\in \mathbb{M}_n$ is called positive semidefinite, if
$x^*Ax\geq0$ for all $x\in \mathbb{C}^n$ (we write $A \geq 0$). If $A$ is invertible and positive semidefinite
matrix, then $A$ is called positive definite (we write $A > 0$). For Hermitian matrices $A, B \in \mathbb{M}_n,$ the inequality $A \geq B$\\
( or $B \leq A$) means that $A-B \geq 0.$
In this paper the set of all positive definite matrices is denoted by $\mathbb{M}_n^+$.\\
For $A\in \mathbb{M}_n,$ the decomposition $A= \Re A + i \Im A$
is called the Cartesian (or Toeplitz) decomposition of $A,$ where $\Re A=\frac{A+A^*}{2}$ and
$\Im A=\frac{A-A^*}{2i}$ are called the real part and imaginary part of $A,$ respectively. \\
Recall that the numerical rang or the field of values of $A\in \mathbb{M}_n$ is defined by
\begin{equation*}
W(A)=\{x^*Ax:x\in\mathbb{C} ^n,x^*x=1\}.
\end{equation*}
It is compact and convex set in $\mathbb{C}$ and contains the spectrum of $A;$ for more information see \cite{Hor}.\\
We say $A \in \mathbb{M}_n$ is accretive, if $\Re A > 0.$ When researching about accretive matrices, it is required to study sectorial type. For this purpose,
let $\mathcal{S}_{\theta}$ denote the sector region in the complex plane $\mathbb{C}$ as follows:
\begin{equation*}
\mathcal{S}_{\theta}=\{z\in \mathbb{C}:\Re z>0, \vert \Im z \vert\leq(\Re z)\tan\theta\},
\end{equation*}
where $0 \leq \theta <\frac{\pi}{2}.$
A matrix $A\in{ \mathbb{M}_n}$ is called a sector matrix, if $W(A)\subset \mathcal{S}_{\theta}$ for some $0 \leq \theta <\frac{\pi}{2}.$
We simply write $A \in \mathcal{S}_{\theta},$ where $W(A)\subset \mathcal{S}_{\theta}.$
Clearly, $A$ is positive definite if and only if $W(A)\subset \mathcal{S}_0$.
We refer the reader to \cite{An, bed1, bed2, pec, Rai, mal} as a sample of articles treating this topic.\\
Recall that the Kronecker(or tensor product) of $A=[a_{ij}] \in \mathbb{M}_{m\times n}$ and $B=[b_{ij}] \in \mathbb{M}_{p\times q}$ denoted by $A \otimes B$ is defined as the following block matrix in $\mathbb{M}_{mp\times nq}$
\[A \otimes B=[ a_{ij}B]=
\begin{bmatrix}
a_{11}B & a_{12}B&\cdots& a_{1n}B\\
a_{21}B & a_{22}B &\cdots& a_{2n}B\\
\vdots&\vdots & \ddots&\vdots\\
a_{m1}B & a_{m2}B &\cdots& a_{mn}B
\end{bmatrix}
.\]

 It is known that the cone of positive semidefinite matrices is closed under the Kronecker product, i.e., if $A \geq 0 $ and $B \geq 0,$ then $A \otimes B \geq 0.$\\
The Hadamard product (or Schur product) of $A, B \in \mathbb{M}_{m\times n} $ is the entrywise product of $A, B$
and denoted by $A \circ B.$\\
It is well known that, see \cite{pec}, if $\{e_j\}_{j=1}^n$ is an orthonormal basis of $\mathbb{C} ^n$, $V: \mathbb{C} ^n\rightarrow \mathbb{C} ^n\otimes \mathbb{C} ^n$ is the isometry $(V^{*}V=I)$ defined by $Ve_j=e_j\otimes e_j,$ then Hadamard product $A\circ B$ regarding to $\{e_j\}_{j=1}^n$ is expressed by
\begin{equation}\label{200}
A\circ B
=V^*(A\otimes B)V.
\end{equation}
By (\ref{200}), we can conclude that the Hadamard product $A \circ B$ is a principal submatrix of the Kronecker product $A \otimes B.$ In addition, $A \otimes B \geq 0$ if and only if $A \circ B \geq 0.$
Also, it can be easily seen that
$$ \Re (A\circ B)
=V^* \Re (A\otimes B)V.$$
Therefore, $\Re (A\otimes B)\geq0$ if and only if $\Re ( A\circ B) \geq0$.

The Kronecker product has many useful and interesting properties (see, e.g.,\cite [Chapter 4]{Hor} ). One of the most important properties of the Kronecer product (the mixed-product property) is as follows:
\begin{equation}\label{306}
(A\otimes B)(C\otimes D)=AC\otimes BD,
\end{equation}
for matrices $A,B,C,D$ with appropriate sizes. Also, for any $A,B\in \mathbb{M}_n^+$
and for any real number $r,$
\begin{equation}\label{L2}
(A\otimes B)^r=A^r\otimes B^r.
\end{equation}
It is known that the set of $m\times n$ matrices become an abelian (commutative) group under the Hadamard product.\\
Suppose $\alpha,\beta\in {\mathbb{C}}$, and $A$ ,$B$ and $C$ are $m\times n$ matrices. Then $C\circ(A+B)=C\circ A+C\circ B$, and
\begin{equation*}
\alpha (A\circ B)=(\alpha A)\circ B= A\circ(\alpha B),
\end{equation*}
consequently
\begin{equation}\label{20}
\alpha A\circ\beta B=\alpha\beta( A\circ B).
\end{equation}
Recent developments on sector matrices can be found in \cite{pec, Raj, Dj}.\\
A linear map $\Phi:\mathbb{M}_n\rightarrow \mathbb{M}_n$ is called positive if
$\Phi(A)\geq0$ where $A\geq0$. In addition, $\Phi$ is said to be unital if $\Phi (I_{n})=I_{n},$ where $I_{n} \in \mathbb{M}_n$ is the identity matrix.\\
Let $J$ be an interval in $\mathbb{R}.$ A continuous function $f:J \longrightarrow \mathbb{R}$ is super multiplicative on $J,$ if $f(xy) \geq f(x)f(y)$ for every $x, y \in J.$ The function $f$ is said to be matrix monotone if $A \geq B$ with spectra are contained in $J,$ implies $f(A) \geq f(B).$ Also, $f$ is called matrix concave if
$$f(\lambda A+(1-\lambda) B) \geq \lambda f(A)+(1-\lambda) f(B)$$
for all $\lambda \in [0,1]$ and for every Hermitian matrices $A, B \in \mathbb{M}_{n} $ whose spectra are in the interval $J.$\\
Notice that if $f$ is a nonnegative continuous function on $[0, \infty),$ then $f$ is matrix monotone if and only if
$f$ is matrix concave, see \cite[Corollary 1.12]{pec}.
Also, we add the following notation
\begin{equation*}
\textit{\textbf{m}}=\lbrace f| f:(0,\infty)\rightarrow(0, \infty)~\text{is an matrix monotone function with}f(1)=1\rbrace.
\end{equation*}
A map $\Phi : \mathbb{M}_n \times \mathbb{M}_n \longrightarrow \mathbb{M}_m$ is said to be jointly convex if
$$\Phi (\lambda A+(1-\lambda)B, \lambda C+(1-\lambda)D ) \leq \lambda \Phi (A, C)+ (1-\lambda) \Phi (B, D)$$
for all $\lambda \in [0, 1]$ and for every $A, B \in \mathbb{M}_n.$\\
A matrix mean $\sigma$ on $\mathbb{M}_n^+$ is a binary operation that satisfying the following conditions:\\
$(i)\, A\leq C$ and $B \leq D$ imply $A\sigma B \leq C\sigma D, $\\
$(ii) \, C^{*}(A\sigma B ) C=(C^{*}AC) \sigma (C^{*} BC),$ for every $C \in \mathbb{M}_n, $\\
$(iii) \, A_{n}\downarrow A$ and $B_{n}\downarrow B$ imply $A_{n}\sigma B_{n}\downarrow A\sigma B, $\\
$(iv) \, I_{n}\sigma I_{n}=I_{n}.$\\
For example, for $A, B \in \mathbb{M}_n^{+},$ the matrix weighted arithmetic, geometric and harmonic means are defined, respectively as follows:\\
$A\nabla _{\nu}B=(1-\nu)A+\nu B, \, A\sharp_{\nu}B= A^{\frac{1}{2}}( A^{-\frac{1}{2}} B A^{-\frac{1}{2}} )^{\nu}A^{\frac{1}{2}}, \, A!_{\nu}B=(A^{-1}\nabla _{\nu} B^{-1})^{-1},$
where $\nu \in [0,1]. $ When $\nu=\frac{1}{2},$ we remove the $\nu$ from the above notations
and for brevity, we write $\nabla, \sharp$ and $!.$\\
For two matrix means $\sigma_{1}$ and $\sigma_{2},$ we say that $\sigma_{1} \leq \sigma_{2} $ if $A\sigma_{1} B \leq A\sigma_{1} B$ for all $ A, B \in \mathbb{M}_n^+.$ In particular, $!_{\nu} \leq \sharp _{\nu} \leq \nabla_{\nu}.$
For matrix mean $\sigma,$ the adjoint of $\sigma$ is denoted by $\sigma^{*}$ and is defined by
$A \sigma^{*}B=(A^{-1}\sigma B^{-1})^{-1},$ for all $ A, B \in \mathbb{M}_n^+.$ Clearly, this is involutive, i.e. $(\sigma^{*})^{*}=\sigma.$\\
The definition of $A\sharp_{\nu}B$ will be extended to accretive matrices.
In fact, for accretive matrices $A, B \in \mathbb{M}$ and $\nu \in \mathbb{R},$
\begin{equation*}
A\sharp_{\nu}B= A^{\frac{1}{2}}( A^{-\frac{1}{2}} B A^{-\frac{1}{2}} )^{\nu}A^{\frac{1}{2}}.
\end{equation*}
For more information and details, we refer the reader to \cite{bed1, bed2, bed3}, as a sample of papers that studying this topic.

 \section{main results}
In this section, we discuss new inequalities for accretive matrices through nonstandard domains.\\
First of all, note that by distributivity of Hadamard multiplication over addition,
for $A,B\in\mathbb{M}_n$,
\begin{align*}
A\circ B&=(\Re A+i \Im A)\circ( \Re B+i \Im B)\\
&=(\Re A\circ \Re B- \Im A \circ \Im B)+i( \Re A\circ \Im B+\Im A\circ \Re B).
\end{align*}
Thus
\begin{equation}\label{de}
\Re (A\circ B)=\Re A\circ \Re B-\Im A\circ \Im B.
\end{equation}
Similarly,
\begin{equation*}
\Re (A\otimes B)=\Re A\otimes \Re B-\Im A\otimes \Im B.
\end{equation*}
It is known that if $A,B\in \mathbb{M}_n^+$, then the Hadamard product $A\circ B$ is also positive. More generally, $A_1\geq A_2\geq 0$ and $B_1\geq B_2\geq 0$ imply
\begin{equation}\label{mon}
A_1\circ B_1\geq A_2\circ B_2\geq 0.
\end{equation}

 Unfortunately, if $A$ and $B$ are accretive, the Hadamard product $A\circ B$ is not accretive, necessarily. For example, let $A=B=1+i$, so $A\circ B=2i$ which is failed to be accretive.

 From above, if $\Im A \circ \Im B \leq 0$, for instance $\Im A\geq0$ and $\Im B\leq0$, then $A\circ B$ is accretive. However, this is not a necessary condition.

 Since $\Re (A\circ B)$ is Hermitian, we can use it in the inequalities, even if it is not positive.
The following result follows immediately from $\eqref{de}$.
\begin{lemma}\label{le17}
Let $A,B\in \mathbb{M}_n$ be accretive matrices. Then
\begin{itemize}
\item If $ \Im A \circ \Im B \geq 0, $ then
\begin{equation}\label{0-17}
\Re (A\circ B) \leq \Re A \circ \Re B.
\end{equation}
\item If $ \Im A \circ \Im B \leq 0, $ then
\begin{equation}\label{1-17}
\Re A \circ \Re B \leq \Re (A\circ B).
\end{equation}
\end{itemize}
\end{lemma}
In \cite[Theorem 2]{chan2}
Chan et al. have given the following inequality for $A,B,C,D\in \mathbb{M}_n^+$ and $\alpha,\beta,r,s>0$ such that $r+s=1$,
\begin{equation}\label{301}
\left(\alpha A+\beta B\right)^r\circ\left(\alpha C+\beta D\right)^s\geq \alpha\left(A^r\circ C^s\right)+\beta\left(B^r\circ D^s\right).
\end{equation}
Next theorem is an extension of inequality $\eqref{301}$,
to achieve our goal,
we will need the following lemma.
\begin{lemma}\cite{cho}
Let $A,B\in \mathcal{S} _{\theta}.$ If $t\in[0,1]$, then
\begin{equation}\label{302}
\cos^{2t}(\theta) \Re A^t\leq \Re^t A\leq \Re A^t
\end{equation}
and if $t\in[-1,0],$ then
\begin{equation}\label{29}
\Re A^t \leq \Re^tA\leq\cos^{2t}(\theta) \Re A^t.
\end{equation}
\end{lemma}

\begin{theorem}\label{t0}
Let $A,B,C,D\in \mathcal{S}_{\theta}$ such that $ \Im A^r \circ \Im C^s \geq 0$ and \\ $ \Im B^r \circ \Im D^s \geq 0,$ where $\alpha,\beta,r,s>0$ and $r+s=1.$ Then
\begin{equation*}
\Re \left( \alpha(A^r\circ C^s)+\beta(B^r\circ D^s) \right)
\leq \sec^2\theta \left(
\Re \left(\alpha A+\beta B\right)^r\circ \Re \left(\alpha C+\beta D\right)^s \right).
\end{equation*}
\end{theorem}
\begin{proof}
Since $r+s=1$ and $r,s>0$ so $r,s\in[0,1],$ it follows that
\begin{align*}
&\Re \left(\alpha A+\beta B\right)^r\circ \Re \left(\alpha C+\beta D\right)^s
\\
&\geq \Re^r\left(\alpha A+\beta B\right)\circ \Re^s\left(\alpha C+\beta D\right)&&(\text{by \eqref{302} and \eqref{mon}})\\
&=\left(\alpha\Re A+\beta \Re B\right)^r\circ\left(\alpha \Re C+\beta \Re D\right)^s\\
&\geq
\alpha\left(\Re^rA\circ \Re^s C\right)+\beta\left(\Re^rB\circ \Re^s D\right)&&(\text{by \eqref{301}})\\
&\geq\alpha\left((\cos^{2r}\theta\Re A^r)\circ(\cos^{2s}\theta\Re C^s)\right)\\
&+\beta\left( (\cos^{2r}\theta\Re B^r)\circ (\cos^{2s}\theta\Re D^s) \right)&&(\text{by \eqref{302}})\\
&=\cos^{2(r+s)}\theta \left( \alpha(\Re A^r\circ\Re C^s)+\beta(\Re B^r\circ \Re D^s) \right)&&(\text{by \eqref{20}})\\
&=\cos^{2}\theta \left( \alpha \left( \Re A^r\circ\Re C^s \right)+\beta \left( \Re B^r\circ \Re D^s \right) \right)\\
&\geq \cos^2\theta \left( \alpha \Re (A^r\circ C^s)+\beta \Re (B^r\circ D^s) \right)&&(\text{by \eqref{0-17}})\\
&=\cos^2\theta \left( \Re \left(\alpha(A^r\circ C^s)+\beta(B^r\circ D^s)\right) \right).
\end{align*}
So, the proof is complete.
\end{proof}
It is worth noting how inequality $\eqref{301}$ is reversed when $r\in(-1,0)\cup(1,2)$.
For this, we need the following lemma.
\begin{lemma}\cite{Raj}\label{L1}
If $A,B\in \mathbb{M}_n$ such that $A \geq 0,$ then the map

 $(A,B)\rightarrow BA^{-1}B$ is jointly convex.
\end{lemma}
\begin{proposition}\label{p1}
If $A,B,C,D\in \mathbb{M}_n^+$ and $\alpha,\beta>0,$ then
\begin{equation}\label{321}
\left(\alpha A+\beta B\right)^r\otimes\left(\alpha C+\beta D\right)^{1-r}\leq \alpha\left(A^r\otimes C^{1-r}\right)+\beta\left(B^r\otimes D^{1-r}\right),
\end{equation}
where $r\in(-1,0)\cup(1,2).$
\end{proposition}
\begin{proof}
First, we prove the assertion for $r\in(-1,0)$. For $r\in(1,2)$ it is sufficient to set $1-r$ instead of $r$, $A$ instead of $C$ and $B$ instead of $D$.\\
For $r\in(-1,0)$ by \cite[p.23 ]{Raj}, we have
\begin{align*}
x^r&=\int^{\infty}_0 (s+x)^{-1}d\mu(s),
\end{align*}
and therefore
\begin{align}\label{s}
(A\otimes B^{-1})^r=\int^{\infty}_0 (s+(A\otimes B^{-1}))^{-1}d\mu(s).
\end{align}
So, by $\eqref{306}$ and $\eqref{L2}$, we get
\begin{align*}
&A^r\otimes B^{1-r}=(A^rI)\otimes B^{-r}B)\\
&=(A^r\otimes B^{-r})(I\otimes B)
=(A\otimes B ^{-1})^r(I\otimes B)\\
&=\int^{\infty}_0(I\otimes B)
\left(sI\otimes I+(A\otimes B^{-1})\right)^{-1}
d\mu(s)&&(\text{by~~\eqref{s}})\\
&=\int^{\infty}_0(I\otimes B)
\left(sI\otimes I+(A\otimes B^{-1})\right)^{-1}
(I\otimes B)^{-1}(I\otimes B)d\mu(s)\\
&=\int^{\infty}_0
(I\otimes B)\left((sI\otimes I)(I\otimes B)+(A\otimes B^{-1})(I\otimes B)\right)^{-1}
(I\otimes B)d\mu(s)\\
&=\int^{\infty}_0(I\otimes B)\left((sI\otimes B)+(A\otimes I)\right)^{-1}
(I\otimes B)d\mu(s).
\end{align*}
Since $(sI\otimes B)+(A\otimes I)$ is positive definite, by Lemma $\ref{L1},$ we have that $(I\otimes B)\left((sI\otimes B)+(A\otimes I)\right)^{-1}
(I\otimes B)$ is jointly convex. Hence, from the viewpoint of the Rieman integral, integrand is also jointly convex, and so is $A^r\otimes B^{1-r}$. This means that for any $A,B,C,D\in \mathbb{M}_n^+$ and scalar $0<\epsilon<1$,
\begin{align*}
\left(\epsilon A+(1-\epsilon)B\right)^r\otimes\left(\epsilon C+(1-\epsilon)D\right)^s\leq\epsilon(A^r\otimes C^s)+(1-\epsilon)(B^r\otimes D^s),
\end{align*}
for $s>0$ and $r+s=1$. Since $0<\alpha/(\alpha+\beta)<1$, by setting $\epsilon=\alpha/(\alpha+\beta)$, we get $\eqref{321}$.
\end{proof}
In the following theorem, we derive inequality $\eqref{321}$ for Hadamard product of positive definite matrices.
\begin{theorem}\label{t1}
Let $A,B,C,D\in \mathbb{M}_n^+$ and $\alpha,\beta>0.$ Then
\begin{equation}\label{308}
\left( \alpha A+\beta B \right)^r \circ \left( \alpha C+\beta D \right)^{1-r} \leq \alpha \left( A^r\circ C^{1-r} \right)+\beta \left( B^r\circ D^{1-r} \right),
\end{equation}
where $r\in(-1,0)\cup(1,2).$
\end{theorem}
\begin{proof}
Define $\Phi:\mathbb{M}_n^+\times
\mathbb{M}_n^+\rightarrow\mathbb{M}^+_{n^2}$ by $\Phi(A,B)=A^r\otimes B^{1-r}$. Recall that the Hadamard product of matrices is the principal submatrix of the Kronecker product of matrices. Consequently, there exists a positive unital linear map $\Psi:
\mathbb{M}^+_{n^2}\rightarrow\mathbb{M}_n^+$ such that $\Psi(A\otimes B)=A\circ B$. Hence,
\begin{align*}
(\Psi\circ\Phi)(A,B)=\Psi(\Phi(A,B))
=\Psi(A^r\otimes B^{1-r})=A^r\circ B^{1-r}.
\end{align*}
Since by Proposition $\ref{p1}$, $\Phi$ is jointly convex and $\Psi$ is positive and linear, the composition $\Psi\circ \Phi$ is also jointly convex. This means that for any $A,B,C,D\in\mathbb{M}_n^+$ and any scalar $0<\epsilon<1$,
\begin{align*}
\left(\epsilon A+(1-\epsilon)B\right)^r\circ\left(\epsilon C+(1-\epsilon)D\right)^{1-r}\leq\epsilon(A^r\circ C^{1-r})+(1-\epsilon)(B^r\circ D^{1-r}).
\end{align*}
Since $0<\alpha/(\alpha+\beta)<1$, by setting $\epsilon=\alpha/(\alpha+\beta)$, we get $\eqref{308}$.
\end{proof}

\begin{lemma}\cite[Theorem 3.1 ]{bed3}
Let $A\in\mathbb{M}_n$ be accretive and let $r\in (1,2)$. Then
\begin{equation}\label{3.1}
\Re A^r\leq \Re ^rA.
\end{equation}
\end{lemma}

In \cite[Theorem 3.3 ]{bed3}, Bedrani et al. have given the following inequality for $A\in \mathcal{S}_{\theta}$ and $B>0$:
\begin{equation}\label{3.3}
\Re A\sharp_r\Re B\leq \sec\theta\Re (A\sharp_r B),
\end{equation}
where $r\in(1,2)$.
Also
in \cite[Theorem 3.9 ]{bed3}, they have given the following inequality for $B\in \mathcal{S}_{\theta}$ and $A>0$:
\begin{equation}\label{3.9}
\Re A\sharp_r\Re B\leq \sec\theta\Re (A\sharp_r B),
\end{equation}
where $r\in(-1,0)$.\\
Now we are ready to present the reversed version of $\eqref{3.1}.$
\begin{lemma}\cite[Theorems 3.2 and 3.8 ]{bed3}
Let $A,B\in \mathbb{M} _n$ be accretive and $r\in (-1,0)\cup(1,2)$. Then
\begin{equation}\label{L}
\Re (A\sharp_r B)\leq \Re A\sharp_r \Re B.
\end{equation}
\end{lemma}
\begin{lemma}\label{L0}
Let $A,B\in \mathcal{S}_{\theta}$. Then for any $r\in (-1, 0) \cup (1, 2),$
\begin{equation*}
\Re A^r\leq \Re^rA\leq\sec\theta\Re A^r.
\end{equation*}
\begin{proof}
If $r\in(-1,0)$ by $\eqref{3.9}$ we have
$$\Re^rA=(I \sharp_r \Re A )\leq\sec\theta
\Re (I\sharp_rA)=sec\theta \Re A^r$$
and by $\eqref{L}$ we have
$$\Re A^r=\Re (I\sharp_rA)\leq(I\sharp_r\Re A)=\Re^rA.
$$
In a similar way, it is proved for $r \in (1, 2),$ by \eqref{3.3} and \eqref{L}.
\end{proof}
\end{lemma}
Next theorem is an extension of $\eqref{308}.$
\begin{theorem}
If $A,B,C,D\in \mathcal{S}_{\theta}$ such that $\Im A^r \circ \Im C^{1-r} \leq 0$ and $\Im B^r \circ \Im D^{1-r} \leq 0$, then
\begin{equation*}
\Re^r\left(\alpha A+\beta B\right)\circ \Re^{1-r}\left(\alpha C+\beta D\right)\leq \sec^2\theta\Re \left(\alpha\left(A^r\circ C^{1-r}\right)+\beta\left(B^r\circ D^{1-r}\right)\right),
\end{equation*}
where $ r\in(-1,0)\cup(1,2)$ and $\alpha,\beta>0.$
\end{theorem}
\begin{proof}
We prove, the assertion for $r\in(-1,0)$. In order to prove it for $r\in(1,2)$, it is
sufficient to put $1-r$ instead of $r$, $A$ instead of $C$ and $B$ instead of $D$.\\
By $\eqref{mon}$, we have
\begin{align*}
&\Re^r\left(\alpha A+\beta B\right)\circ\Re^{1-r}\left(\alpha C+\beta D\right)\\
&=\left(\alpha \Re A+\beta \Re B\right)^r\circ\left(\alpha \Re C+\beta \Re D\right)^{1-r}\\
&\leq
\alpha\left(\Re^rA\circ \Re^{1-r}C\right)+\beta\left(\Re^rB\circ \Re^{1-r}D\right)&&(\text{by~~Theorem ~\ref{t1}})\\
&\leq\alpha\left((\sec\theta \Re A^r)\circ(\sec\theta \Re C^{1-r})\right)\\
&+\beta\left((\sec\theta\Re B^r)\circ(\sec\theta\Re D^{1-r})\right)&&(\text{by~~Lemma \ref{L0} and \eqref{mon}}) \\
&=\sec^{2}\theta\left(\alpha(\Re A^r\circ\Re C^{1-r})+\beta(\Re B^r\circ\Re D^{1-r})\right)&&(\text{by \eqref{20}})
\\
&\leq\sec^{2}\theta(\alpha\Re \left( A^r\circ C^{1-r}\right)+\beta\Re \left( B^r\circ D^{1-r})\right)&&(\text{by \eqref{1-17}})\\
&=\sec^{2}\theta\left(\Re \left(\alpha(A^r\circ C^{1-r})+\beta(B^r\circ D^{1-r}\right)\right).
\end{align*}
So, the proof is complete.
\end{proof}
\section{ Matrix monotone functions and matrix means}
In this section, we discuss some new inequalities for accretive matrices. In particular, we provide accretive versions of many known results for positive definite matrices.\\
We notice that if $\Phi$ is a positive unital linear map and $A \in \mathbb{M}_n,$ then
\begin{equation*}
\Phi(\Re A)=\Re (\Phi(A)),
\end{equation*}
and if $A,B \in \mathbb{M}_n^+,$ then
\begin{equation}\label{061}
\Phi(A\circ B)=\Phi(A)\circ\Phi(B).
\end{equation}
\begin{remark}
Similar to $\eqref{061}$, for matrices $A,B\in \mathcal{S}_{\theta}$, the equality
\begin{equation*}
\Re \Phi(A\circ B)=\Re \left(\Phi(A)\right)\circ\Re \left(\Phi(B)\right),
\end{equation*}
does not hold generally. For example, let \[ \Phi(X)=X,~
A=
\begin{bmatrix}
1-i & 1+i\\
-1+i & 1+i
\end{bmatrix},
B=
\begin{bmatrix}
1+i & 1+i\\
-1+i & 1-i
\end{bmatrix}
.\]
Then by a simple calculation, we obtain
\begin{equation*}
\Re \Phi(A\circ B)>\Re \left(\Phi(A)\right)\circ\Re \left(\Phi(B)\right).
\end{equation*}
\end{remark}
Now, we have an accretive version of \eqref{061} as follows:
\begin{proposition}
If $A,B\in \mathbb{M} _n$ are accretive such that $\Im A\circ \Im B\leq0$ and $\Phi$ is a positive unital linear map, then
\begin{equation}\label{62}
\Re \Phi(A\circ B)\geq\Re (\Phi( A))\circ\Re (\Phi( B)).
\end{equation}
\end{proposition}
\begin{proof}
Using \eqref{1-17},
\begin{align*}
\Re \Phi(A\circ B)&=\Phi(\Re (A\circ B))&&\\
&\geq\Phi(\Re A\circ\Re B)&&\\
&=\Phi(\Re A)\circ\Phi(\Re B)&& \\
&=\Re (\Phi( A))\circ\Re (\Phi( B)).&&
\end{align*}
So, the proof is complete.
\end{proof}
Now, by applying the following lemmas, we intend to present some Hadamard inequalities involving matrix monotone functions.
\begin{lemma}\cite{bed1}
If $ A \in \mathcal{S}_{\theta} $ and $f\in \textit{\textbf{m}}$, then
\begin{equation}\label{24}
f(\Re A)\leq\Re (f(A))\leq\sec^2\theta f(\Re A).
\end{equation}
\end{lemma}
\begin{lemma}\cite{pec}
Let $A, B \geq 0 $ and $\Phi$ be a positive unital linear map. If $f$ is a super-multiplicative matrix concave function on $(0,\infty),$ then
\begin{equation}\label{39}
f \left(\Phi( A\circ B) \right) \geq \Phi \left( f(A)\circ f(B) \right).
\end{equation}
\end{lemma}
In the following theorem, inequality $(\ref{39})$ is extended.


\begin{theorem}\label{t4}
Let $A,B\in \mathcal{S}_{\theta}$ such that $\Im A\circ \Im B\leq0$ and $\Phi$ be a positive unital linear map. If $f$ is a super-multiplicative matrix concave function on $(0,\infty)$, then
\begin{equation*}
\Re \left( \Phi(f(A)) \right) \circ \Re \left( \Phi(f(B)) \right) \leq \sec^4\theta \Re \left( f(\Phi( A\circ B) \right).
\end{equation*}
\end{theorem}
\begin{proof}
Applying \eqref{24} and \eqref{39}, we have
\begin{align*}
\cos^4\theta[\Re \left(\Phi(f(A))\right)&\circ\Re \left(\Phi(f(B))\right)]\\
&=\cos^2\theta\Re \left(\Phi(f(A))\right)\circ\cos^2\theta\Re \left(\Phi(f(B))\right)&&\\
&=\cos^2\theta\Phi\left(\Re (f(A))\right)\circ\cos^2\theta\Phi\left(\Re (f(B))\right)&&\\
&\leq\Phi\left(f(\Re (A))\right)\circ\Phi\left(f(\Re (B))\right)&&\\
&=\Phi\left(f(\Re (A))\circ f(\Re (B))\right)&&\\
&\leq f(\Phi(\Re A\circ\Re B))&& \\
&\leq f\left(\Phi(\Re (A\circ B))\right)&& \\
&= f\left(\Re (\Phi(A\circ B))\right)&&\\
&\leq \Re \left(f (\Phi(A\circ B))\right),
\end{align*}
completing the proof.
\end{proof}
In studying matrix means, it is customary to compare between different means that arise from different matrix monotone function. In Theorem $\ref{K}$, we present such comparison for sectorial matrices.

 In \cite{kho}, the authors have proved the next result.
\begin{proposition}
Let $A,B\in {\mathcal S}_{\theta}$ such that $0<mI \leq \Re A, \Re B \leq MI$ for some scalars $m < M.$ If $\sigma$ is an arbitrary matrix mean and $\sigma_1,\sigma_2$ are two matrix means between $\sigma$ and $\sigma^*$, then
\begin{equation}\label{m1}
\cos^2\theta\Phi \Re (A\sigma_1 B)+mM\Phi^{-1}\Re (A\sigma_2B) \leq (M+m)I,
\end{equation}
for every positive unital linear map $\Phi$.
\end{proposition}
Recall that a norm $ \Vert . \Vert$ on $\mathbb{M}_n$ is saied to be unitarily invariant if $\Vert UAV \Vert= \Vert A \Vert$ for all $U, V, A \in \mathbb{M}_n$ where $U, V$ are unitary matrices.\\
Next, we have the following lemma for Hadamard product.
\begin{lemma}\cite{sin9}
If $A,B \in\mathbb{M}_n^+ $ and $t \in [0, 1],$ then
\begin{equation}\label{25}
2\|A^{\frac{1}{2}}\circ B^{\frac{1}{2}}\|\leq\|A^t\circ B^{1-t}+A^{1-t}\circ B^t\|\leq\|A+B\|,
\end{equation}
for all unitarily invariant norm $\|.\|.$
\end{lemma}
Combining \eqref{m1}, and \eqref{25}, we get the following theorem. The proof is straightforward and omitted.
\begin{theorem}\label{K}
Let $A,B\in \mathcal{S}_{\theta}$ such that $0< mI\leq \Re A, \Re B\leq MI$ for some scalars $m < M$ and $\Im A\circ \Im B\leq0.$ If $\sigma_1$ and $\sigma_2$ are two matrix means between $\sigma,\sigma^*$ for some matrix mean $\sigma$,
then for every positive unital linear map $\Phi$, we get
\begin{equation*}
\| \Phi^{\frac{1}{2}}(\Re (A\sigma_1 B))\circ\Phi^{-\frac{1}{2}}(\Re (A\sigma_2 B))\|\leq \dfrac{M+m}{2\sqrt{mM}}\sec\theta.
\end{equation*}
\end{theorem}
To obtain a relationship between the inverse of Hadamard product of two matrices and Hadamard product of their reverses, we need the following lemma that has been stated in \cite[P. 204]{zha1}.
\begin{lemma}\cite{zha1}\label{lf}
Let $A\in\mathbb{M}_n^+$ and $X \in \mathbb{M}_{n\times m} $ such that $X^*X=I_m.$ Then
\begin{equation}\label{41-0}
(X^*AX)^{-1}\leq X^*A^{-1}X\leq\frac{(\lambda+\mu)^2}{4\lambda\mu}(X^*AX)^{-1},
\end{equation}
where $\lambda$ is the largest and $\mu$ is the smallest eigenvalue of $A$.

 Specially, if $A,~B \in \mathbb{M}_n^+ $ choosing appropriate X and replacing $A$ with $A\otimes B$, we get
\begin{equation}\label{41}
(A\circ B)^{-1}\leq A^{-1}\circ B^{-1}\leq\frac{(\lambda+\mu)^2}{4\lambda\mu}(A\circ B)^{-1},
\end{equation}
where $\lambda$ is the largest and $\mu$ is the smallest eigenvalue of $A\otimes B$.
\end{lemma}
\begin{remark}
It is easy to show that the function $k(h)=\dfrac{(1+h)^{2}}{4h},$ the well known Kantorovich function, is increasing on $[1, \infty).$ Hence, if \\ $ 0 < m\leq \mu \leq \lambda \leq M$ and we set $h_{1}=\frac{\lambda}{\mu}, \, h_{2}=\frac{M}{m},$
then $\frac{(\lambda+\mu)^2}{4\lambda\mu}=k(h_{1}) \leq k(h_{2})=\frac{(M+m)^2}{4M m}.$ Thus, in Lemma \ref{lf}, if $mI\leq A\leq MI$, we can replace coefficient$\frac{(\lambda+\mu)^2}{4\lambda\mu}$ with $\frac{(M+m)^2}{4M m}$.
\end{remark}
Next, we present the sectorial version of $\eqref{41}$.
\begin{theorem}\label{t2}
Let $A,B\in \mathcal{S}_{\theta}$ such that $\Im A\circ \Im B\leq0$. Then
\begin{align*}
\cos^4\theta\Re (A\circ B)^{-1}\leq \Re A^{-1}\circ \Re B^{-1}.
\end{align*}
\end{theorem}
\begin{proof}
By $(\ref{41})$, we have
\begin{align*}
\Re (A\circ B)^{-1}&\leq \Re^{-1}(A\circ B)&& \\
&\leq (\Re A\circ \Re B)^{-1}&&\\
&\leq \Re^{-1} A\circ \Re^{-1} B&&\\
&\leq \sec^2\theta \Re A^{-1}\circ \sec^2\theta \Re B^{-1}&& \\
&=\sec^4\theta(\Re A^{-1}\circ \Re B^{-1}).&&
\end{align*}
So, the proof is complete.
\end{proof}
The following theorem is a reverse version of Theorem $\ref{t2}$.
\begin{theorem}\label{t3}
Let $A,B, A\circ B \in \mathcal{S}_{\theta}$ such that either
\begin{itemize}
\item[1)]
$\Im A^{-1}\circ \Im B^{-1}\leq0$ and $mI \leq \Re(A\otimes B) \leq 	MI $ for some scalars $0<m \leq M,$
\item[or]
\item[2)] $\Im A\circ \Im B\geq0$ and $ 	mI \leq \Re A\otimes \Re B \leq 	MI $ for some scalars $0 < m \leq M.$
\end{itemize}
Then
$$\Re A^{-1}\circ \Re B^{-1}\leq \sec^2\theta \frac{(M+m)^2}{4mM}\Re (A\circ B)^{-1}.$$
\end{theorem}
\begin{proof}
For the first condition, in view of \eqref{200} and \eqref{L2}, there exists an isometry $X$ of appropriate size such that $A^{-1} \circ B^{-1}=X^*(A\otimes B)^{-1}X,$ then using \eqref{1-17}, \eqref{29} and \eqref{41-0} respectively,\\
we have
\begin{align*}
\Re A^{-1}\circ \Re B^{-1}&\leq
\Re ( A^{-1}\circ B^{-1})=\Re (X^*(A\otimes B)^{-1}X)\\&=X^*\Re (A\otimes B)^{-1}X\\
&\leq X^*\Re^{-1}(A\otimes B)X\\
&\leq \frac{(M+m)^2}{4M m}(X^*\Re (A\otimes B)X)^{-1}\\
&= \frac{(M+m)^2}{4M m}\Re ^{-1}(X^*(A\otimes B)X)\\
&= \frac{(M+m)^2}{4M m}\Re^{-1}(A\circ B)\\
&\leq \sec^2\theta \frac{(M+m)^2}{4M m}\Re (A\circ B)^{-1}.
\end{align*}
To prove the result under the second condition, applying \eqref{29}, \eqref{41} and \eqref{0-17} sequentially, we get
\begin{align*}
\Re A^{-1}\circ \Re B^{-1}& \leq \Re^{-1} A\circ \Re^{-1} B &&\\
&\leq \frac{(M+m)^2}{4M m} (\Re A\circ\Re B)^{-1}&& \\
&\leq \frac{(M+m)^2}{4M m} \Re^{-1}( A\circ B)&&\\
&\leq \sec^2\theta \frac{(M+m)^2}{4M m}\Re (A\circ B)^{-1}.
\end{align*}
So, the result holds.
\end{proof}


\end{document}